\documentclass[twoside,11pt]{article}
\usepackage{graphicx,amscd,amsfonts,amsmath,amsthm,amssymb,latexsym,amsfonts,color}
\usepackage{epsfig,float,epstopdf,array,bm,cite}
\usepackage{cases}
\usepackage{array}
\usepackage{multirow}
\usepackage[table]{xcolor}
\usepackage{relsize}
\usepackage{graphicx}
\usepackage{float}
\usepackage[bookmarksnumbered, colorlinks,plainpages]{hyperref}
\def\diag{\textrm{diag}}

\def\h{\hspace{-0.2cm}}

\newtheorem{theorem}{\bf Theorem}
\newtheorem{lemma}{\bf Lemma}

\newtheorem{example}{\bf Example}

\begin{document}
\title{\bf A block triangular preconditioner for a class of three-by-three block saddle point problems}
\author{\small\bf  Hamed Aslani$^\dag$, Davod Khojasteh Salkuyeh$^{\ddag\dag}$\thanks{\noindent Corresponding author. \newline 
		Emails:	hamedaslani525@gmail.com (H. Aslani), khojasteh@guilan.ac.ir (D.K. Salkuyeh)}\\[2mm]
	\textit{{\small $^\dag$Faculty of Mathematical Sciences, University of Guilan, Rasht, Iran}} \\
	\textit{{\small $^\ddag$Faculty of Mathematical Sciences, and Center of Excellence for Mathematical Modelling,}}\\
	\textit{{\small Optimization and Combinational Computing (MMOCC), University of Guilan, Rasht, Iran}}  \\
}

\date{}
\maketitle
\vspace{-0.5cm}

\noindent\hrulefill\\
{\bf Abstract.}   This paper deals with solving a class of three-by-three block saddle point problems. The systems are solved by preconditioning techniques. Based on an iterative method, we construct a block upper triangular preconditioner. The convergence of the presented method is studied in details. Finally, some numerical experiments are given to demonstrate the superiority of the proposed preconditioner over some existing ones. \\[-3mm]


\noindent\hrulefill\\

\pagestyle{myheadings}\markboth{H. Aslani, D.K. Salkuyeh}{A block triangular preconditioner for three-by-three block saddle point problems}
\thispagestyle{empty}

\section{Introduction}
We are concerned with the following three-by-three block system of linear equations
\begin{equation}\label{eq1}
\mathcal{A} {\bf x} \equiv\left(\begin{array}{ccc}
{A} & {B^{T}} & {0} \\
{B} & {0} & {C^{T}} \\
{0} & {C} & {0}
\end{array}\right)\left(\begin{array}{l}
{x} \\
{y} \\
{z}
\end{array}\right)=\left(\begin{array}{l}
{f} \\
{g} \\
{h} 
\end{array}\right),
\end{equation}
where $A\in \mathbb{R}^{n\times n}$ is a symmetric positive definite (SPD), $B\in \mathbb{R}^{m\times n}$ and  $C\in \mathbb{R}^{l\times m}$ have full row rank,   $f\in \mathbb{R}^n$, $g\in \mathbb{R}^m$ and $h\in \mathbb{R}^l$ are known, and ${\bf x}=\left(x; y; z\right)$ is an unknown vector to be determined. We use $\left(x; y; z\right)$ to denote the vector $(x^{T},y^{T},z^{T})^{T}.$ Throughout the paper, we assume that $n \geq m$ and $m \geq l.$ These hypothesis  guarantee the nonsingularity of \eqref{eq1}, see \cite{A2} for further details. So, the solution of \eqref{eq1} exists and is unique. In this case, the coefficient matrix of the system \eqref{eq1} is of order $\bf{n}$, in which ${\bf{n}}=n+m+l.$


Evidently, one can solve the equivalent linear linear system instead of the original system: 
\begin{equation}\label{eq2}
\mathcal{B} {\bf x} \equiv\left(\begin{array}{ccc}
{A} & {B^{T}} & {0} \\
{-B} & {0} & {-C^{T}} \\
{0} & {C} & {0}
\end{array}\right)\left(\begin{array}{l}
{x} \\
{y} \\
{z}
\end{array}\right)=\left(\begin{array}{l}
{f} \\
{-g} \\
{h}
\end{array}\right)={\bf b}.
\end{equation}
Although $\mathcal{B}$ loses symmetry, it retains some  noteworthy properties:\\
\noindent 1. $\mathcal{B}$ is semipositive real, that is, $v^{T} \mathcal{B}  v > 0$, for all $v \in \mathbb{R}^{n}$;\\
\noindent 2. $\mathcal{B}$  is positive semistable which means that $\Re(\lambda) \geq 0$ for all $\lambda \in \sigma(\mathcal{B} ) $, where $\sigma(\mathcal{B})$ 
denotes the spectrum of $\mathcal{B}$.\\ 
\noindent These properties are so important for Krylov subspace methods like GMRES (see \cite {Benzi1,Saad}).

Systems of linear equations with the form \eqref{eq1} are called three-by-three saddle point problems, which appears in many engineering applications, such as the least squares problems \cite{A7}, the Karush-Kuhn-Tucker (KKT) conditions of a type of quadratic programming \cite{KKT}, the discrete finite element methods for solving time-dependent Maxwell equation with discontinuous coefficient \cite{A1,A3,A5} and so on.

The stationary iterative methods usually combined with the acceleration techniques, be-
cause they may fail to converge or converge too slowly. The acceleration techniques, such as Chebyshev or Krylov subspace methods, while very successful, have some limitations. For instance, the use of Krylov acceleration require the computation of an orthonormal basis for the Krylov subspace, which may to have an adverse impact on the efficiency of these methods, like GMRES. There are some alternative acceleration techniques investigated by researchers, which we do not discuss here.

The coefficient matrix $\mathcal{A}$ in Eq. \eqref{eq1} can be viewed as a standard block saddle point problem of the form
\begin{equation}\label{part1}
\mathcal{A} = \left( {\begin{array}{cc|c}
	A & {B^T } & {0}  \\
	B & 0 & C^T   \\
	\hline
	0 & C & 0  \\
	\end{array}} \right),
\end{equation} 
or 
\begin{equation}\label{part2}
\mathcal{A} = \left( {\begin{array}{c|cc}
	A & {B^T } & {0}  \\
	\hline
	B & 0 & C^T   \\
	0 & C & 0  \\
	\end{array}} \right).
\end{equation}
Since the attributes of the submatrix in \eqref{part1} and \eqref{part2} are different from the standard saddle point problems, many preconditioning strategies in the literature for standard two-by-two saddle point problems can not be directly  applied for solving \eqref{eq1}, for instance, shift-splitting preconditioners \cite{Beik-SIAM,SS1,SS2,SS3,SS4,SS5,SS6,SalkuyehNA}, block triangular preconditioners \cite{BT1,BT2,BT3,BT4,BT5} and parameterized preconditioners \cite{PT}. In recent years, the iterative solution of the three-by-three saddle point problems has attracted substantial attention. Recently, Abdolmaleki et al. \cite{BD} proposed the following block diagonal preconditioner
 \begin{equation}\label{eqBD1}
 \mathcal{P}_{D1}=\left(\begin{array}{ccc}
 {A} & {0} & {0} \\
 {0} & {\alpha I + \beta B B^{T}} & {0} \\
 {0} & {0} & {\alpha I + \beta CC^{T}}
 \end{array}\right),
 \end{equation}
where $\alpha, \beta >0.$ They also discussed  properties of the corresponding iteration matrix ${\mathcal{P}}_{D1} ^{-1} \mathcal{B}.$ In \cite{Huang1}, the following preconditioner was applied for accelerating the convergence rate of Krylov subspace method
\begin{equation}\label{eq4}
\mathcal{P}_{D2}=\left(\begin{array}{ccc}
{A} & {0} & {0} \\
{0} & {S} & {0} \\
{0} & {0} & {C S^{-1} C^{T}}
\end{array}\right),
\end{equation}
where $S=B A^{-1} B^{T}.$ The preconditioner $\mathcal{P}_{D1}$ recived wide attention. Xie and Li \cite{A2} introduced  the following preconditioners
\[
{\cal P}_1=\begin{pmatrix}
A &   0  & 0 \\
B &  -S  & C^T \\
0 &   0  & CS^{-1}C^T 
\end{pmatrix},~
{\cal P}_2=\begin{pmatrix}
A &   0  & 0 \\
B &  -S  & C^T \\
0 &   0  & -CS^{-1}C^T 
\end{pmatrix},~
{\cal P}_3=\begin{pmatrix}
A &   B^T  & 0 \\
B &  -S    & 0 \\
0 &   0    & -CS^{-1}C^T 
\end{pmatrix},
\]
\small{These three} block preconditioners lead to the corresponding preconditioned matrices $\mathcal{P}_{1} ^{-1} \mathcal{A}, \mathcal{P}_{2} ^{-1} \mathcal{A}$ and $\mathcal{P}_{3} ^{-1} \mathcal{A},$ that have only eigenvalues $\{1\},\{ \pm \frac{ 1}{2} , 1\}$ and $\{\pm 1\},$ respectively. Numerical results in \cite{BD} confirmed the robustness of the preconditioner $\mathcal{P}_{D1},$ for solving \eqref{eq2}. In this work, a development of the block diagonal preconditioner $\mathcal{P}_{D1}$ is employed. This new preconditioner is induced using a splitting of the coefficient matrix in \eqref{eq2}. The corresponding splitting iteration method and its convergence properties are given.  

The rest of  paper is arranged as follows. Section \ref{sec2} is devoted to introduce and convergence analysis of the proposed method. Furthermore, implementation issues of the corresponding preconditioner  are briefly discussed. Numerical experiments are presented in Section \ref{sec3}. Finally, in Section \ref{sec4} some
concluding remarks are given.

Throughout the paper, $I$ stands for the identity matrix of suitable order. $x^{H}$ indicates the conjugate transpose of  any arbitrary complex vectors $x.$ For a given matrix $A$ with real eigenvalues, $\lambda_{\min}$ and $\lambda_{\max}$ stand for the minimum and maximum eigenvalue of $A,$ respectively. Moreover, the notations $\sigma(A)$ and $\rho(A)$ denote the set of all eigenvalues of $A$ and the spectral radius of $A,$ respectively. The minimum and maximum singular value of  $A$ are represented by $\sigma_{\min}$ and $\sigma_{\max}$, respectively. 
\section{Preconditioner and convergence analysis}\label{sec2}
We first split the coefficient martix in \eqref{eq2} as $\mathcal{B}=\mathcal{P}-\mathcal{R},$ where  
$$\mathcal{P}=\left(\begin{array}{ccc}
	A & B^{T} & 0 \\
	0 & \alpha I+ \beta B B^{T} & -C^{T} \\
	0 & 0 & \alpha I+ \beta C C^{T}
\end{array}\right), \quad \mathcal{R}=\left(\begin{array}{ccc}
0 & 0 & 0 \\
B & 0 & -C^{T} \\
0 & -C & 0
\end{array}\right),$$
in which $\alpha$ and $\beta$ are given positive constants. Evidently, the matrix $\mathcal{P}$ is nonsingular. So, the iterative scheme associated with the splitting $\mathcal{B}=\mathcal{P}-\mathcal{R},$ can be constructed as 
\begin{equation}\label{split}
{\bf{x}}^{(k+1)}= \mathcal{G}_{\alpha,\beta} {\bf{x}}^{(k)}+ f, \qquad k=0,1,2,\dots,
\end{equation}
where $\bf{x}^{(0)}$ is arbitrary and $\mathcal{G}_{\alpha,\beta}= \mathcal{P}^{-1} \mathcal{R}$ is the iteration matrix and $f=\mathcal{P}^{-1} {\bf {b}}.$
 
 In the sequal, we investigate the convergence properties of the proposed  iterative method for solving the double saddle point problem \eqref{eq2}. To do so, we need to recall a result about the evaluation of the roots of a quadratic equations as follows.
\begin{lemma}\label{lem1}
\cite{Roots} Consider the quadratic equation $x^{2} - b x +c=0,$ where $b $ and $c$ are real
	numbers. Both roots of the equation are less than one in modulus if and only if $|c| < 1$ and $|b| < 1+c.$
\end{lemma}
\begin{theorem}\label{th1}
	Suppose that $A\in \mathbb{R}^{n \times n}$ is SPD, $B \in \mathbb{R}^{m \times n}$ and $C \in \mathbb{R}^{m \times l}$ are full row rank matrices. Then, the iterative method
	\eqref{split} converges to the uniqe solution of \eqref{eq2} for any initial guess, if
	\begin{equation}\label{assump}
	\frac{{{\sigma_{\max}}^{2} (C)}}{\alpha+\beta {\sigma_{\min}}^{2}(C^{T})} +2 \frac{ {{\sigma_{\max}}^{2} ( B^{T})}}{\lambda_{\min} (A)} < 4 (\alpha +\beta {\sigma_{\min}}^{2} ( B^{T})).
	\end{equation}  
\end{theorem}
\begin{proof}
	Assume that $(\lambda; \bf{x})$ is an eigenpair of the iteration matrix $\mathcal{G}_{\alpha,\beta},$ where ${\bf{x}}:=(x;y;z)$. So, we have $\mathcal{G}_{\alpha,\beta} {\bf{x}} = \lambda {\bf{x}}$ which is equivalent to
	say that
\begin{numcases}{}
\lambda (A x + B^{T} y)=0, \label{eq21}\\ 
\lambda ((\alpha I + \beta B B^{T}) y -C^{T} z)=Bx + (\alpha I + \beta B B^{T}) y,
\label{eq22}\\
\lambda (\alpha I + \beta C C^{T}) z=-C y +(\alpha I + \beta C C^{T}) z .\label{eq23}
\end{numcases}
\noindent If $\lambda=0,$ then there is nothing to prove. So, we assume that $\lambda \neq 0.$  We claim that $y \neq 0.$ If not, from \eqref{eq21} we have $A x=0.$ Since $A$ is a SPD matrix, we deduce that $x=0.$ Hence, from \eqref{eq22}  and the assumption that $C$ has full row rank we conclude that $z=0.$ Therefore, ${\bf{x}}=0$ and it is contrary to the assumption that ${\bf{x}}$ is an eigenvector. 

Furthermore, we assert that $\lambda \neq 1.$ Otherwise, the Eqs. \eqref{eq22}, \eqref{eq23} and \eqref{eq24} are reduced to
\begin{eqnarray} 
x &=& -A^{-1} B^{T} y,\label{eq24}\\
C^{T} z &=& -Bx,\label{eq25}\\
y^{H} C^{T} &=& 0,\label{eq26}
\end{eqnarray}
respectively. Pre-multiplying  Eq. \eqref{eq25}  by $y^{H}$ and substituting \eqref{eq26} into it, gives $y^{H} B x=0.$ This along with  \eqref{eq24} leads to $y^{H} B A^{-1} B^{T} y=0,$ equivalently, $(B^{T} y)^{H}  A^{-1} (B^{T} y)=0.$ In view of the positive definitness of $A,$ we get $B^{T}y=0.$ Then, since $B$ is of full row rank, we deduce that $y=0,$ which is impossible.  

In the following, we assume that $\lambda \neq 0,1$ and $ y \neq 0$. Without loss of generality, we assume that $||y||_{2} = 1$.  From \eqref{eq21} and \eqref{eq23}, we obtain
\begin{eqnarray}
x &=& -A^{-1} B^{T} y,\label{eq27} \\
z &=& \frac{1}{1-\lambda} (\alpha I+\beta C C^{T})^{-1} C y. \label{eq28}
\end{eqnarray}
\noindent Substituting the above relations into \eqref{eq22}, yields
$$\lambda \left(\left(\alpha I+\beta B B^{T}\right)y+\frac{1}{\lambda-1} C^{T} (\alpha I+\beta C C^{T})^{-1} Cy\right)=-B A^{-1} B^{T} y+\left(\alpha I+\beta B B^{T}\right)y.$$
By multiplying both sides of the preceding equality on the left by $\lambda-1$ and $y^{H}$ and with some algebra, we obtain the following quadratic equation
$$\lambda ^{2} -\frac{b}{a} \lambda + \frac{c}{a}=0,$$
where 
$$a=y^{H} (\alpha I+\beta B B^{T}) y, \quad c=y^{H} (\alpha I+\beta B B^{T}) y-y^{H} B A^{-1} B^{T} y,$$
$$ b=2 y^{H} (\alpha I+\beta B B^{T}) y -y^{H} C^{T}  (\alpha I+\beta C C^{T})^{-1} Cy - y^{H} B A^{-1} B^{T} y.$$
According to Lemma \ref{lem1}, the following  inequalities
 \begin{equation}\label{eq30}
 \left|\frac{c}{a}\right|<1, \qquad \left|\frac{b}{a}\right|< 1+\frac{c}{a},
 \end{equation}
  imply $|\lambda|<1.$ Clearly, whenever the inequality 
\begin{equation}\label{eq31}
y^{H} B A^{-1} B^{T} y< 2 y^{H} (\alpha I+\beta B B^{T}) y,
\end{equation}
holds, the first inequality of \eqref{eq30} is on. On the other hand, by easy manipulations we can observe that the second relation of \eqref{eq30} holds, if
\begin{equation}\label{eq32}
P_{\alpha,\beta}:=y^{H} C^{T} (\alpha I+\beta C C^{T})^{-1} Cy +2y^{H} B A^{-1} B^{T} y< 4 y^{H} (\alpha I+\beta B B^{T}) y =:Q_{\alpha,\beta}.
\end{equation}
Notice that, inequality \eqref{eq31} is ensured when \eqref{eq32} holds true. Hence, Eq. \eqref{eq31} is ignored. We first assume that $w:=Cy\neq 0$ (note that  $v=B^Ty\neq 0$).
 According to Courant-Fisher inequality \cite{Saad}  we have  
\begin{eqnarray} 
P_{\alpha,\beta} &=& 
 \frac{w^{H} (\alpha I+\beta C C^{T})^{-1} w}{ w^Hw} \frac{y^HC^TCy}{y^Hy}  +2 \frac{v^{H}  A^{-1} v}{v^Hv}  \frac{y^{H}  BB^T y}{y^Hy} \nonumber   \\ 
 &\leq& \lambda_{\max} (\alpha I+\beta C C^{T})^{-1} \lambda_{\max} (C^T C)+2 \lambda_{\max} (A^{-1})
 \lambda_{\max} (BB^T) \nonumber\\
 &=& \frac{{{\sigma_{\max}}^{2} (C)}}{\alpha+\beta {\sigma_{\min}}^{2}(C^{T})} +2  \frac{{{\sigma_{\max}}^{2} ( B^{T})}}{\lambda_{\min} (A)}. \label{EqPab}
 \end{eqnarray}
It is necessary to mention that the upper bound for $P_{\alpha,\beta}$ given above  is valid even if $w=0$. On the other hand, we have
\begin{equation}\label{EqQab}
Q_{\alpha,\beta}= y^{H} (\alpha I +\beta B B^{T}) y \geq  \alpha +\beta {\sigma_{\min}}^{2} ( B^{T}).
\end{equation}
Now, from the Eqs. \eqref{EqPab} and \eqref{EqQab} we deduce that if the inequality \eqref{assump} holds true, then the convergence of the proposed method is deduced. 
\end{proof}

Since both of the matrices $B$ and $C$ are of full row rank, we deduce that ${\sigma_{\min}} ( B^{T}),{\sigma_{\min}} ( C^{T})>0$. Hence, it follows from Eq. \eqref{assump} that for a large enough value of $\alpha$ or $\beta$ the method is convergent. However, for large values of $\alpha$ and $\beta$ the corresponding preconditioner may be inefficient. In the sequel we propose a method for choosing suitable.

Let  
$$P=2  \frac{{{\sigma_{\max}}^{2} ( B^{T})}}{\lambda_{\min} (A)}.$$ 
Based on Theorem \ref{th1}, a sufficient condition for convergence of the proposed method is as follows
\begin{equation}\label{eq33}
\frac{{{\sigma}_{\max}}^{2} (C)}{\alpha \left(1+\frac{\beta}{\alpha} {{\sigma}_{\min}}^{2} (C^{T})\right)} +P <4 \alpha \left(1+\frac{\beta}{\alpha} {{\sigma}_{\min}}^{2} (B^{T})\right).
\end{equation}
Now, if $\beta \geq \alpha$ and 
\begin{equation}\label{eq34}
\frac{{{\sigma}_{\max}}^{2} (C)}{\alpha (1+ {{\sigma}_{\min}}^{2} (C^{T}))} +P <4 \alpha (1+ {{\sigma}_{\min}}^{2} (B^{T})),
\end{equation}
then the inequality \eqref{eq33} holds true. By a little algebra, we can rewrite \eqref{eq34} as the following quadratic inequality

\begin{equation}\label{eq35}
q(\alpha):=-4 (1+{\sigma_{\min}}^{2} (B^{T})) (1+{\sigma_{\min}}^{2} (C^{T})) \alpha^{2} + P (1+{\sigma_{\min}}^{2} (C^{T})) \alpha + {\sigma_{\max}}^{2} (C) < 0.
\end{equation}  
Notice that the coefficient of $\alpha^2$ in the polynomial $q$ is negative and $q(0)={\sigma_{\max}}^{2} (C)>0$.
Therefore, the polynomial $q$ has two real roots, one negative and a positive. The positive one is given by
   $$\tilde{\alpha} = \frac{ P\left(1+{\sigma_{\min}}^{2} (C^{T})\right) + \sqrt{\Delta}} {8(1+{\sigma_{\min}}^{2} (B^{T}))(1+{\sigma_{\min}}^{2} (C^{T}))},$$
where $\Delta=P^{2} (1+{\sigma_{\min}}^{2} (C^{T}))^{2} +16 (1+{\sigma_{\min}}^{2} (B^{T})) \left(1+{\sigma_{\min}}^{2} (C^{T})\right) {\sigma_{\max}}^{2} (C) .$ 

According to the above results, we can claim that if\\
\noindent (i) $\beta \geq \alpha$,\\
\noindent (ii) $\alpha > \tilde{\alpha},$\\
then $q(\alpha)<0$ and the proposed method is convergent for any initial choice of $\bf x^{(0)},$ i.e., $\rho (\mathcal{G}_{\alpha, \beta})<1.$

Based on the above results, the eigenvalues of $\mathcal{G}_{\alpha,\beta}$ are contained in a circle centered at origin with radius 1. In addition, we obviously have
$$\mathcal{P}^{-1} \mathcal{B}=I - \mathcal{G}_{\alpha,\beta}.$$
So, the eigenvalues of $\mathcal{P}^{-1} \mathcal{B}$ included in a circle centered $(1,0)$ with radius 1. Therefore, $\mathcal{P}$ serves a preconditioner for a Krylov subspace methods such as GMRES.

We end this section by applying the preconditioner $\mathcal{P}$ within the Krylov subspace methods to solve the system $\mathcal{B} \bf{x}=\bf{b}.$ In each iteration, we need to compute vectors of the form $v=\mathcal{P}^{-1} w,$ equivalently, $w=\mathcal{P} v.$ Now, by taking $v = (v_{1}; v_{2}; v_{3})$ and $w = (w_{1}; w_{2}; w_{3})$  the following algorithm can be given:

\noindent\hrulefill

\noindent Algorithm 1: Computation  of $(v_1;v_2;v_3)=\mathcal{P} ^{-1}(w_1;w_2;w_3)$.

\noindent\hrulefill\\
\noindent 1. Solve $(\alpha I+\beta C C^{T}) v_{3}= w_{3}$ for $v_{3};$ \\
\noindent 2. Solve $(\alpha I+\beta B B^{T}) v_{2}= w_{2}+C^{T} v_{3}$ for $v_{2};$ \\
\noindent 3. Solve $A v_{1}=w_{1}-B^{T} v_{2}$ for $v_{1}$.

\noindent\hrulefill

In each step of this algorithm a system of linear equations  should be solved. Since the coefficient matrices of these systems are SPD, they can be solved exactly using the Cholesky factorization or inexactly using the conjugate gradient (CG) method. In practice, in Step 1 of algorithm it is recommended to choose the values of $\alpha$ and $\beta$ such that (See \cite{BT2,Estrin})
\[
\beta=\alpha \frac{1}{\|C\|_2^2}.
\] 
In the same way to choose the values of $\alpha$ and $\beta$ in Step 2 satisfying
\[
\beta=\alpha \frac{1}{\|B\|_2^2}.
\] 
However, since $\alpha$ and $\beta$ are in common in Steps 1 and 2 we propose to use
\begin{equation}\label{betaave}
\beta=\frac{\alpha}{2} \left(\frac{1}{\|C\|_2^2}+ \frac{1}{\|B\|_2^2}\right),
\end{equation}
for both of the steps. We will shortly see in the section of the numerical results that a small value of $\alpha$ along with the value of $\beta$ using \eqref{betaave} give usually suitable results.

\section{Numerical results}\label{sec3}

In this section, we give some numerical experiments to illustrate the superiority
of the proposed preconditioner $\mathcal{P}$ over the  recently suggested ones in the literature. At each iteration of  the preconditioners $\mathcal{P}_{D1}, \mathcal{P}_{D2}$ and $\mathcal{P}_{1},$ three linear subsystems with SPD coefficient matrices should be solved. These subsystems are solved by the CG method.

In our numerical experiments, the iteration is started from a zero vector
and terminated as soon as 
$$Res=\frac{\left\|{\mathbf{b}}-\mathcal{A} {\bf{x}}^{(k)}\right\|_{2}}{\left\|{\mathbf{b}}\right\|_{2}} \leq 10^{-6},$$
where ${\bf{x}}^{(k)}$ is the computed solution at iteration $k.$ The maximum number of
iterations is set to be 1000. We have used the right-hand side vector $\bf{b}$ 
such that the exact solution is a vector of all ones. For the inner CG iterations, the iteration is terminated as soon as the residual norm is reduced by a factor of $10^{3}$. In addition, the maximum number of inner iterations is set to be 100.  For all the test problems, we set $S=B (\diag (A))^{-1} B^{T}.$ 

In the following, we will compare the preconditioners from aspects of the number of total iteration steps (denoted by ``IT"), and elapsed CPU times in seconds (denoted by ``CPU"). As well as, the accuracy of the methods are compared under 
\[
Err=\frac{\left\|{\bf{x}}^{(k)}-\bf{x}^{*}\right\|}{\left\|\bf{x}^{*}\right\|},
\]
where ${\bf{x}}^{(k)}$ and $\bf{x}^{*}$ stand for the current iteration and the exact solution of \eqref{eq2}, respectively.  The symbols  $``\dag"$ and $``\ddag"$ show that the method has not converged in 1000 seconds and $maxit$, respectively. Also, by $``\S"$ we mean that the coefficient matrix $\mathcal{B}$ does not satisfy the assumptions:\\
\noindent (i) $A$ is a SPD matrix, \\
\noindent (ii) $B$ and $C$ are full row rank matrices.

 All the computations are implemented in \textsc{Matlab} R2019a on a Laptop with intel (R) Core(TM) i5-8265U CPU @ 1.60 GHz 8.GB.  
\begin{example}\label{ex1}\rm
	 Consider the saddle point problem \eqref{eq2} with (see \cite{Huang1,A2})
	\begin{align*}
	A=\left(\begin{array}{cc}
	{I \otimes T+T \otimes I} & {0} \\
	{0} & {I \otimes T+T \otimes I}
	\end{array}\right) \in \mathbb{R}^{2 p^{2} \times 2 p^{2}},
	\end{align*}
	$
	B=(I \otimes F \quad F \otimes I) \in \mathbb{R}^{p^{2} \times 2 p^{2}}$ and  $C=E \otimes F \in \mathbb{R}^{p^{2} \times p^{2}}
	$
	where
	\begin{align*}
	T=\frac{1}{h^{2}} \cdot \operatorname{tridiag}(-1,2,-1) \in \mathbb{R}^{p \times p}, \quad F=\frac{1}{h} \cdot \operatorname{tridiag}(0,1,-1) \in \mathbb{R}^{p \times p},
	\end{align*}
	and $
	E=\operatorname{diag}\left(1, p+1, 2p+1, \ldots, p^{2}-p+1\right)
	$ in which $\otimes$ denotes the Kronecker product and $h={1}/{(p+1)}$ stands for the discretization meshsize.
For the preconditioner $\mathcal{P}$, we set $\alpha=5\times 10^{-2}$ and compute $\beta$ using \eqref{betaave}. These values are listed in Table \ref{tab1}. 
\begin{table}[!t]
	\centering
	\caption{\small{The values of $\beta$ involved in the preconditioner $\mathcal{P}$ for Example \ref{ex1} with $\alpha=10^{-2}$ .}\label{tab4}}\vspace{0.25cm}
	\begin{tabular}{|p{1cm}| p{2.1cm} p{2.1cm}p{2.1cm} p{2.1cm}p{2.1cm}| }
		\hline
		\multirow{1}{*}{$p$}   & 16            &   32   &  64   &  128   &  256\\
		\hline
		\multirow{1}{*}{$\beta$}   & 0.94            &   1.83   &  3.60   &  7.14   &  14.22\\
		\hline
	\end{tabular}
		\label{tab1}
	\centering
	\caption{Numerical results for Example \ref{ex1}.\label{tab2}}\vspace{0.25cm}
	\begin{tabular}{|p{1.5cm}| p{1.5cm}|p{1.7cm}p{1.7cm} p{1.7cm}p{1.7cm}p{1.7cm}| }
		\hline
		\multirow{1}{*}{Precon.}   & $p$            &   16   & 32   &  64   &  128&256\\
		\hline\hline
		
		\multirow{4}{*}{$I$} & IT  & 425 &949&\ddag&\ddag&\ddag\\
		& CPU & 0.72 &8.97 &55.10&161.23&637.21\\
		& Res & 8.6e-07&9.9e-07 &2.7e-03&6.7e-03&4.9e-02\\
		& Err & 2.6e-06&2.4e-5 &1.8e-01&5.5e-01&7.8e-01\\
		\hline
		\multirow{4}{*}{$\mathcal{P} $}&
		IT&33&42&53&75&141\\
		& CPU&0.03&0.11&0.75&4.05&70.81\\
		&Res&8.7e-07&5.9e-07&6.2e-07&8.9e-07&8.8e-07\\
		& Err &1.6e-06 &1.6e-06&6.6e-07&2.2e-05&1.9e-05\\
		\hline
		\multirow{4}{*}{\small{$\mathcal{P}_{D1}$}\tiny{(case(i))}} &
		IT&109&80&65&89&191\\
		& CPU&0.17&0.33&1.89&6.91&113.10\\
		&Res&7.0e-07&8.4e-07&7.8e-07&8.6e-07&9.7e-07\\
		& Err & 3.0e-07 &6.0e-07&1.2e-06&2.3e-05&4.3e-05\\
		\hline
		\multirow{4}{*}{\small{$\mathcal{P}_{D1}$}\tiny{(case(ii))}} &
		IT&49&53&69&103&181\\
		& CPU&0.08&0.14&1.10&6.50&102.66\\
		&Res&3.9e-07&6.6e-07&2.8e-07&9.6e-07&9.6e-07\\
		& Err & 2.1e-07 &1.8e-07&8.2e-06&2.5e-05&4.0e-05\\
		\hline
		\multirow{4}{*}{$\mathcal{P}_{1}$}&
		IT&114&466&\ddag&-&-\\
		& CPU&0.75&24.66&421.56&\dag&\dag\\
		&Res&8.10e-07&2.7e-06&4.6e-02&-&-\\
		& Err & 1.9e-06&2.1e-06&4.2e-01&-&-\\
		\hline
		\multirow{4}{*}{$\mathcal{P}_{D2}$}&
		IT&170&792&\ddag&-&-\\
		& CPU&1.20&48.91&406.2&\dag&\dag\\
		& Res&9.2e-07&1.5e-04&5.7e-02&-&-\\
		& Err & 1.0e-6&1.8e-05&1.4e-01&-&-\\
		\hline
	\end{tabular}
	\label{tab2}
\end{table}	

We observe from Table \ref{tab1} that $\beta \geq \alpha,$ which is in agreement with what we claimed in Section \ref{sec2}. We consider two choices for parameters $\alpha$ and $\beta$ in the preconditioner $\mathcal{P}_{D1}$ as the following cases:\\
\noindent Case (i): $\alpha=10^{-3}$ and $\beta=1,$ as considered in \cite{BD};\\
\noindent Case (ii): According to the Table \ref{tab1}.\\
 \noindent Numerical results of the flexible GMRES (FGMRES) method \cite{Saad,FGMRES}  in conjunction with the   preconditioners  for solving the double saddle point problem \eqref{eq2} are presented in Table \ref{tab2}. These results clearly show that the preconditioner $\mathcal{P}$ is quite effective. In this problem, we find that the overall computation times and the iteration numbers for the preconditioner $\mathcal{P}$ is less than the other examined  preconditioners.  
\end{example}
\begin{example}\label{ex2} \rm
	We consider the three-by-three block saddle point problem \eqref{eq1} for which ( see \cite{Huang1,A2})
	\begin{align*}
	A=\operatorname{blkdiag}\left(2 W^{T} W+D_{1}, D_{2}, D_{3}\right) \in \mathbb{R}^{n \times n},
	\end{align*}
	is a block-diagonal matrix,
	\begin{align*}
	B=\left[E,-I_{2 \widetilde{p}}, I_{2 \widetilde{p}}\right] \in \mathbb{R}^{m \times n} \text { and } \quad C=E^{T} \in \mathbb{R}^{l \times m},
	\end{align*}
	are both full row-rank matrices where $\tilde{p}=p^{2},$
	$\hat{p}=p(p+1);$ $W=(w_{i,j}) \in \mathbb{R}^{\hat{p} \times \hat{p}}$ with $w_{i,j}=e^{-2((i/3)^{2}+(j/3)^{2}};$ $D_{1}=I_{\widehat{p}}$ is an identity matrix; $D_{i}=\operatorname{diag}(d_{j}^{(i)}) \in \mathbb{R}^{2\tilde{p} \times 2\tilde{p}},$ $i=2,3,$ are diagonal matrices, with
	\begin{eqnarray*}
		d_{j}^{(2)}&\h=\h&\left\{\begin{array}{ll}
			{1,} & {\text { for } \quad 1 \leq j \leq \tilde{p}}, \\
			{10^{-5}(j-\tilde{p})^{2}},  & \text { for }   {\tilde{p}+1 \leq j \leq 2 \tilde{p}},
		\end{array}\right. \\
		d_{j}^{(3)} &\h=\h&10^{-5}(j+\tilde{p})^{2}
		\text { for } 1 \leq j \leq 2 \tilde{p},
	\end{eqnarray*}
	and
	\begin{align*}
	E=\left(\begin{array}{cc}
	{\widehat{E} \otimes I_{p}} \\
	{I_{p} \otimes \widehat{E}}
	\end{array}\right), \quad \widehat{E}=\left(\begin{array}{ccccc}
	{2} & {-1} & {} & {} & {}\\
	{} & {2} & {-1} & {} & {} \\
	{} & {} & {\ddots} & {\ddots} & {}\\
	{} & {} & {} & {2} & {-1}
	\end{array}\right) \in \mathbb{R}^{p \times(p+1)}.
	\end{align*}
	According to the above definitions, we have $n=\hat{p}+4\tilde{p}$, $m=2\tilde{p}$ and $l=\hat{p}$.
	
	The parameters $\alpha$ and $\beta$ involved in the preconditioner $\mathcal{P}_{D1}$ are chosen as $\alpha=10^{-1}$ and $\beta=1$ (See \cite{BD}). Also, in the preconditioner $\mathcal{P},$ we set $\alpha=5 \times 10^{-1}$ and $\beta$ is computed similar to Example \ref{ex1}, that are listed in Table \ref{tab33}. In Table \ref{tab4}, we give numerical results for the FGMRES method incorporated with the preconditioners $\mathcal{P}, \mathcal{P}_{D1}, \mathcal{P}_{1}$ and $\mathcal{P}_{D2}.$ Hence, we have also reported the results of FGMRES without preconditioning. As observed, the preconditioner $\mathcal{P}$ substantially accelerate the convergence rate of FGMRES. It should be mentioned that when $p$ is large, only $\mathcal{P}$ and $\mathcal{P}_{D1}$ are feasible in practice.
	\begin{table}[!t]
		\centering
		\caption{\small{The values of $\beta$ involved in the preconditioner $\mathcal{P}$ for Example \ref{ex2} with $\alpha=5 \times 10^{-1}$}.\label{tab33}}\vspace{0.25cm}
		\begin{tabular}{|p{.9cm}| p{1.8cm} p{1.8cm}p{1.8cm}p{1.8cm} p{1.8cm}p{1.8cm}| }
			\hline
			\multirow{1}{*}{$p$}   & 16            &   32   &  64   &  128   &  256&512\\
			\hline
			\multirow{1}{*}{$\beta$}   & 0.36            &   0.35   &  0.35   &  0.35   &  0.35&0.34\\
			\hline
		\end{tabular}
		\label{tab3}
	\centering
	\caption{Numerical results for Example \ref{ex2}.\label{tab4}}\vspace{0.25cm}
	\begin{tabular}{|p{1.5cm}| p{1.5cm}|p{1.5cm}p{1.5cm}p{1.5cm} p{1.5cm}p{1.5cm}p{1.5cm}| }
		\hline
		\multirow{1}{*}{Precon.}   & $p$            &   16   & 32   &  64   &  128&256&512\\
		\hline\hline
		
		\multirow{4}{*}{$I$} & IT  & 186 &190&187&180&-&-\\
		& CPU & 0.25 &1.04 &3.01&27.87&\dag&\dag\\
		& Res & 1.0e-06&9.9e-07 &1.0e-06&9.8e-07&-&-\\
		& Err & 1.3e-06&1.4e-5 &1.4e-05&1.4e-05&-&-\\
		\hline
		\multirow{4}{*}{$\mathcal{P}$}&
		IT&53&55&56&54&52&50\\
		& CPU&0.06&0.15&0.55&2.89&14.50&57.30\\
		&Res&8.2e-07&9.3e-07&9.9e-07&9.9e-07&8.9e-07&8.4e-07\\
		& Err &1.2e-05 &1.5e-05&1.5e-05&1.6e-05&1.6e-05&1.2e-06\\
		\hline
		\multirow{4}{*}{$\mathcal{P}_{D1}$}&
		IT&70&69&68&65&63&60\\
		& CPU&0.11&0.28&0.93&5.46&24.08&99.58\\
		&Res&1.0e-06&9.5e-07&8.8e-07&9.3e-07&8.5e-07&9.5e-07\\
		& Err & 5.6e-06 &5.7e-06&5.0e-6&5.5e-06&4.9e-06&5.2e-06\\
		\hline
		\multirow{4}{*}{$\mathcal{P}_{1}$}&
		IT&10&10&10&9&-&-\\
		& CPU&0.04&0.14&1.96&41.00&\dag&\dag\\
		&Res&2.8e-07&3.1e-07&2.4e-07&8.8e-07&-&-\\
		& Err & 1.1e-06&2.1e-06&8.5e-07&1.7e-06&-&-\\
		\hline
		\multirow{4}{*}{$\mathcal{P}_{D2}$}&
		IT&19&19&19&18&-&-\\
		& CPU&0.05&0.24&3.51&63.50&\dag&\dag\\
		& Res&2.7e-07&2.0e-07&4.4e-07&8.7e-07&-&-\\
		& Err & 7.7e-07&4.7e-07&1.6e-06&3.7e-06&-&-\\
		\hline
	\end{tabular}
		\label{tab4}
\end{table}	
\end{example}
\begin{example}\label{Ex3}\rm
	Consider the quadratic program \cite{Huang2,KKT}:\\
\begin{eqnarray}
&\h\h& \min_{x\in\Bbb{R}^n,y\in\Bbb{R}^l} \frac{1}{2}x^TAx+r^Tx+q^Ty \label{CUTER}\\
\nonumber &\h\h&~~~s.t.:~~~ Bx+C^Ty=b, 
\end{eqnarray}
where the vector $\lambda\in\Bbb{R}^{m}$ is the Lagrange multiplier. To solve the above problem we define the Lagrange function
	\[
L(x, y, \lambda) =\frac{1}{2}x^TAx+r^Tx+q^Ty+\lambda^T(Bx+C^Ty-b),
\]
where the vector $\lambda\in\Bbb{R}^{m}$ is the Lagrange multiplier. Then the Karush-Kuhn-Tucker necessary conditions of
\eqref{CUTER}  are as follows (see \cite{Bertsekas})
\[
\nabla_{x} L(x, y, \lambda)=0,\quad \nabla_{y} L(x, y, \lambda)=0 \quad \text{and} \quad \nabla_{\lambda} L(x, y, \lambda)=0.
\]
These equations lead to a system of linear equations of the form \eqref{eq1}. In this example, the matrices $A, B$ and $C$ have been chosen from the CUTEr collection \cite{CUTER}. We note that for the test matrix MOSARQP1, the matrix $C$ is not full row rank. So, the matrix $C S^{-1} C^{T}$ is symmetric positive semidefinite. This means that the preconditioners $\mathcal{P}_{1}$ and $\mathcal{P}_{D2}$ are singular. Consequently, $\mathcal{P}_{1}$ and $\mathcal{P}_{D2}$ can not be applied as a preconditioner.  Similarly, for the test matrices AUG2D and AUG2DC, $A$ is symmetric positive semidefinite. Accordingly, the matrix $S$ and as well as $\mathcal{P}_{1}$ and $\mathcal{P}_{D1}$ can not be formed.\\
\indent In this example, for the preconditioners $\mathcal{P}$ and $\mathcal{P}_{D1}$ we set $\alpha=5 \times 10^{-1},$ and $\beta$ computed according to the formula \eqref{betaave}, that are reported in Table \ref{Tab5}. The result for FGMRES and application of the preconditioners are shown in Table \ref{Tab6}. As seen in Table \ref{Tab6}, the iteration steps and computational time for $\mathcal{P}$ are less than the other ones.
 \begin{table}[!t]
 	\centering
 	\caption{\small{The values of $\beta$ involved in the preconditioner $\mathcal{P}$ for Example \ref{Ex3} with $\alpha=5 \times 10^{-1}$ .}\label{Tab5}}\vspace{0.25cm}
	\begin{tabular}{| p{1.6cm}| p{2cm}p{2cm} p{1.8cm} p{1.7cm} p{2cm}p{2cm}| }
		\hline
		\multirow{1}{*}{ Matrix}   & MOSARQP1            &   AUG2DC   &  AUG2D   &  YAO  &  LISWET12&HUESMOD\\
		\hline
		\multirow{1}{*}{$\beta$}   & 0.66            &   0.82   &  0.53   &  0.60   &  0.60&0.49\\
		\hline
	\end{tabular}
	\label{tab6}
	\centering	
	\caption{Numerical results for Example \ref{Ex3}. \label{Tab6}}\vspace{0.25cm}
\begin{tabular}{| p{1.5cm}| p{1.5cm}|p{2cm} p{1.4cm} p{1.3cm} p{1.4cm}p{1.7cm}p{1.9cm}|} \hline
	\multirow{2}{*}{Precon.}& Matrix        & MOSARQP1    & AUG2DC     & AUG2D   &  YAO & LISWET12 &\small{HUESMOD} \\
	&{$\bf{n}$}& 5700&50400&50400&6004&30004&20002 \\
	&$nnz$& 14434&140600&140200&18006&90006&70000 \\
	 \hline\hline
	\multirow{4}{*}{$I$} & IT            & 110        & 69         &  69        &  61 &56 &9      \\
	& CPU           &0.24    & 0.65       &  0.63     &  0.08&0.28&0.01     \\
	& Err           & 9.9e-07  & 8.4e-07   &  8.5e-07  &  9.0e-07&8.3e-07& 6.2e-10 \\
	& Res          & 3.9e-06  & 3.8e-06   &  3.8e-06  &  4.1e-06 &4.1e-06& 6.3e-10\\ \hline
	\multirow{4}{*}{$\mathcal{P}$} & IT            & 31       & 32        &  46        &  35 &33&7       \\
	& CPU           & 0.04      & 0.26       &  0.54      &  0.06&0.19 &0.02    \\
	& Err           & 9.8e-07  &8.5e-07   &  9.4e-07  &  8.8e-07&8.7e-07 &5.1e-09\\ 
	& Res          & 5.2e-06  & 3.2e-06   &  4.2e-06  &  3.8e-06 &3.7e-06&5.2e-09\\ \hline
	\multirow{4}{*}{${\mathcal{P}}_{D1}$} & IT            & 60        & 55         & 58         & 49  &46&11    \\
	& CPU           & 0.11      & 0.61       & 0.61       &  0.09  &0.32 &0.04  \\
	& Err           & 8.4e-07  & 8.6e-7    & 9.0e-07   &  9.1e-07&8.0e-07 &8.1e-08\\ 
	& Res          & 4.1e-06  & 3.6e-06   &  3.9e-06  &  4.5e-06 &4.0e-06&1.4e-07\\ \hline
	\multirow{4}{*}{$\mathcal{P}_{1}$}    & IT            &    \S   & \S         & \S  &  \dag &\dag &9      \\
	& CPU           & -      & -       &  -      &  - &-&10.65    \\
	& Err          & -  &-   & -  &  - &-&3.0e-06\\
	& Res          & -  & -   & -  &  - &-&3.0e-06\\
	\hline
	\multirow{4}{*}{$\mathcal{P}_{D2}$}     & IT            & \S        & \S        &  \S      &  \dag  &\dag   &10   \\
	& CPU           & -      & -       &  -      &  -   &- &10.48 \\
	& Err           & -  & -    & -  & - &-&1.8e-06\\
	& Res      & -  & -   &-  &  - &-&1.7e-06\\
	\hline
\end{tabular}
\end{table}
 \end{example}
\section{Conclusion}\label{sec4}
	We have proposed a new iteration method for solving a class of three-by-three saddle point problems. The convergence theory of the method have been studied. The exploited preconditioner from the presented method, has been applied for accelerating the convergence rate of Krylov subspaces method, especially for GMRES method. The remarkable point was that introduced preconditioner is easy to implement.
	Numerical results indicate that the presented preconditioner is effective.   


\begin{thebibliography}{23}
	
\bibitem{BD}  M. Abdolmaleki, S. Karimi, D.K. Salkuyeh, A new block diagonal preconditioner for a class of 3$\times$3 block saddle point problems, Mediterranean Journal of Mathematics, 19 (2022) 43. 

\bibitem {A1} F. Assous, P. Degond, E. Heintze, P.A. Raviart, J. Segre, On a finite-element method for solving the three-dimensional Maxwell equations, J. Comput. Phys. 109 (1993) 222-237.

\bibitem{Beik-SIAM} F. P. A. Beik and M. Benzi, Iterative methods for double saddle point systems, SIAM J. Matrix Anal. Appl. 39 (2018), pp. 902-921. 

\bibitem{BT1} F. P. A. Beik and M. Benzi, Block preconditioners for saddle point systems arising from liquid crystal directors modeling, CALCOLO 55 (2018) 29.

\bibitem{Benzi1}  M. Benzi, Preconditioning techniques for large linear systems: A survey, J. Comput. Phys. 182 (2002) 418-477.

\bibitem{BT2} M. Benzi, G. H. Golub and J. Liesen, Numerical Solution of Saddle Point Problems, Acta Numer. 14 (2005) 1-137.

\bibitem{BT3} M. Benzi, M.K. Ng, Q. Niu, Z. Wang, A relaxed dimensional factorization preconditioner for the incompressible Navier-Stokes equations, J. Comput. Phys. 230 (2011) 6185-6202.

\bibitem{Bertsekas} D.P. Bertsekas, Nonlinear Programming, 2nd Ed., Athena Scientic,1999.

\bibitem{BT4} Z.-H. Cao, Positive stable block triangular preconditioners for symmetric saddle point problems, Appl. Numer. Math. 57 (2007) 899-910.

\bibitem{SS1} Y. Cao, J. Du, Q. Niu, Shift-splitting preconditioners for saddle point problems, J. Comput. Appl. Math. 270 (2014) 239-250.

\bibitem{SS2} Y. Cao, Sen Li, L. Yao, A class of generalized shift-splitting preconditioners for nonsymmetric saddle point problems, Appl. Math. Lett. 49 (2015) 20-27.

\bibitem{A3} Z.-M. Chen, Q. Du, J. Zou, Finite element methods with matching and nonmatching meshes for Maxwell equations with discontinuous coefficients, SIAM J. Numer Anal. 37 (2000) 1542-1570.

\bibitem{SS3} C.-R. Chen, C.-F. Ma, A generalized shift-splitting preconditioner for singular saddle point problems, Appl. Math. Comput. 269 (2015) 947-955.

\bibitem{A5} P. Ciarlet, J. Zou, Finite element convergence for the Darwin model to Maxwell’s equations, RAIRO Math. Modelling Numer. Anal. 31 (1997) 213-249.

\bibitem{Estrin} R. Estrin, C. Greif, Towards an optimal condition number of certain augmented
Lagrangian-type saddle-point matrices, Numer. Linear Algebra Appl. 23 (2016) 693-705.

\bibitem{BT5} H.C. Elman, D.J. Silvester, A.J. Wathen, Performance and analysis of saddle point preconditioners for the discrete steady-state Navier-Stokes equations, Numer. Math. 90 (2002) 665-688.

\bibitem{CUTER} N.I.M. Gould, D. Orban, P.L. Toint, CUTEr and SifDec, a constrained and
unconstrained testing environment, revisited, ACM Trans. Math. Softw. 29 (2003),
373-394.
	
\bibitem{KKT} D.R. Han, X.M. Yuan, Local linear convergence of the alternating direction method of multipliers for quadratic programs, SIAM J. Numer. Anal. 51 (2013) 3446-3457.

\bibitem{Huang2} N. Huang, Variable parameter Uzawa method for solving a class of block three-by-three saddle point problems, Numer. Algor. 85 (2020), 1233-1254.
	
\bibitem{Huang1} N. Huang, C.-F. Ma, Spectral analysis of the preconditioned system for the 3$\times$3 block saddle point problem, Numer. Algor. 81 (2019) 421-444.

\bibitem{PT} Y.-F. Ke, C.-F. Ma, The parameterized preconditioner for the generalized saddle point problems from the incompressible Navier- Stokes equations, J. Comput. Appl. Math. 37 (2018) 3385-3398.


\bibitem{Saad} Y. Saad, Iterative methods for sparse linear systems, PWS Press, New York, 1995.

\bibitem{FGMRES} Y. Saad, A flexible inner-outer preconditioned GMRES algorithm, SIAM Journal on Scientific Computing 14 (1993) 461--469.

\bibitem{SS4} D.K. Salkuyeh, M. Rahimian, A modification of the generalized shift-splitting method for singular saddle point problems, Comput. Math. Appl. 74 (2017) 2940-2949.

\bibitem{SS5} D.K. Salkuyeh, M. Masoudi, D. Hezari, On the generalized shift-splitting preconditioner for saddle point problems, Appl. Math. Lett. 48 (2015) 55-61.

\bibitem{SalkuyehNA} D.K. Salkuyeh, M. Masoudi, A new relaxed HSS preconditioner for saddle point problems. Numer Algor 74 (2017) 781--795. 

\bibitem{SS6} Q.-Q. Shen, Q. Shi, Generalized shift-splitting preconditioners for nonsingular and singular generalized saddle point problems, Comput. Math. Appl. 72 (2016) 632-641.

\bibitem{A2} X. Xie, H.-B. Li, A note on preconditioning for the $3 \times 3$ block saddle point problem, Comput. Math. Appl. 79 (2020) 3289--3296.	

\bibitem{Roots} D. M. Young, Iterative Solution or Large Linear Systems, Academic Press, New York, 1971.
	
\bibitem{A7} J.-Y. Yuan, Numerical methods for generalized least squares problems, J. Comput. Appl. Math. 66 (1996) 571--584.

\end{thebibliography}
\end{document}